\documentclass{amsart}
\usepackage[utf8]{inputenc}
\usepackage{geometry}
\usepackage{amsmath}
\usepackage[english]{babel}
\usepackage[T1]{fontenc}
\usepackage{amsfonts}
\usepackage{float}
\usepackage{amssymb}
\usepackage[toc,page]{appendix}
\usepackage{longtable}
\usepackage{amsthm}
\theoremstyle{plain}
\usepackage[overload]{empheq}
\usepackage{enumerate}
\usepackage{mathtools}
\usepackage{array}
\usepackage{graphicx}
\usepackage{hyperref}
\usepackage{todonotes}
\usepackage[all]{xy}
\usepackage{tikz-cd}
\usepackage{multirow}

\usepackage{flafter}

\newcommand{\N}{\mathbb{N}}
\newcommand{\w}{\omega}
\newcommand{\OO}{\mathcal{O}}

\newcommand{\C}{\mathbb{C}}
\newcommand{\Z}{\mathbb{Z}}
\newcommand{\Q}{\mathbb{Q}}
\newcommand{\R}{\mathbb{R}}
\newcommand{\PP}{\mathbb{P}}
\newcommand{\sign}{\mathrm{sign}}
\newcommand{\rank}{\mathrm{rank}}
\newcommand{\Hom}{\mathrm{Hom}}

\newcommand{\Aut}{\mathrm{Aut}}

\newcommand{\rk}{\mathrm{rk}\,}
\newcommand{\disc}{\mathrm{disc}}

\newcommand{\Gr}{\mathbb{G}\mathrm{r}}

\newtheorem{thm}{Theorem}[section]
\newtheorem{thmINTRO}{Theorem}

\newtheorem{lem}[thm]{Lemma}
\newtheorem{prop}[thm]{Proposition}

\newtheorem{definition}[thm]{Definition}
\newtheorem{rem}[thm]{Remark}

\theoremstyle{definition}

\title{Irreducible Holomorphic Symplectic Manifolds with an action of $\Z_3^4 : \mathcal A_6$}
\date{}

\author[P. Comparin]{Paola Comparin \textsuperscript{1}}
\address{\textsuperscript{1}
Departamento de Matem\'atica y Estad\'istica, Universidad de la Frontera, Av.
Francisco Salazar 1145, Temuco, Chile}
\email{paola.comparin@ufrontera.cl}

\author[R. Demelle]{Romain Demelle \textsuperscript{2}}
\address{\textsuperscript{2}
Universit\'{e} de Poitiers, Laboratoire de Math\'{e}matiques et Applications, UMR
7348 du CNRS, 11 bd Marie et
Pierre Curie, 86073 Poitiers Cedex 9,
France}
\email{romain.demelle@univ-poitiers.fr}

\author[P. Quezada Mora]{Pablo Quezada Mora \textsuperscript{3}}
\address{\textsuperscript{3}
Departamento de Matem\'atica y Estad\'istica, Universidad de la Frontera, Av.
Francisco Salazar 1145, Temuco, Chile}
\email{pabloquezada04@gmail.com}

\counterwithout{table}{section}

\begin{document}
\maketitle
\begin{abstract}
Höhn and Mason classified the possible symplectic groups acting on an Irreducible Holomorphic Symplectic (IHS) manifold of K3$^{[2]}$-type, finding that $\Z_3^4 : \mathcal A_6$ is the symplectic group with the biggest order. 
In this paper, we study the possible IHS manifolds of K3$^{[2]}$-type with a symplectic action of $\Z_3^4 : \mathcal A_6$ and also admitting a non-symplectic automorphism. We characterize such IHS manifolds. In particular we prove that there exists an IHS manifold of K3$^{[2]}$-type with finite automorphism group of order 174960, the biggest possible order for the automorphism group of a IHS manifold of K3$^{[2]}$-type, and it is the Fano variety of lines of the Fermat cubic fourfold.
\end{abstract}

\section*{Introduction}
An Irreducible Holomorphic Symplectic (IHS) manifold is a complex compact Kähler smooth manifold which is simply connected and admits a unique (up to scalar) nowhere degenerate holomorphic 2-form $\omega_X$. IHS manifolds always have even dimension. The 2-dimensional IHS manifolds are K3 surfaces while in dimension greater than two there are four known examples of IHS manifolds up to deformation: Hilbert schemes of $n$ points on a K3 surface, called  of K3$^{[n]}$-type, generalized Kummer surfaces 
and O'Grady's examples in dimension six and ten. 
Given a finite group $G$ acting on an IHS manifold $X$ we can study the action induced by an element $g \in G$ on the 2-form $\w_X$. Since $H^{2,0}(X)$ is one-dimensional, $g$ acts by multiplication by a non-zero complex scalar $\alpha(g)\in\C^*$.
The automorphism $g$ is said to be \emph{symplectic} if it acts trivially on $\omega_X$ and \emph{non-symplectic} otherwise. Thus, given a finite group $G$ acting on $X$, we can consider the following exact sequence
\[
\xymatrix{
  1 \ar[r] & G_0 \ar[r] & G  \ar[r]^-{\alpha} & \mu_m  \ar[r] & 1 
}
\]
where the group $G_0$ is the kernel of $\alpha$, i.e. the group of automorphisms acting symplectically on $X$. 

Given an IHS manifold $X$, one of its most important properties is that the second cohomology group $H^2(X,\Z)$ is a lattice. In the case of K3 surface, the lattice $H^2(X,\Z)$ is a unimodular lattice and in his seminal paper \cite{NikulinGroups} Nikulin used this fact to study groups acting on K3 surfaces by the induced action on $H^2(X,\Z)$. In particular, he classified the finite abelian groups acting faithfully and symplectically on a K3 surface. Later, following Nikulin's steps, Mukai in \cite{Mukai} studied finite groups $G$ which act faithfully and symplectically on a K3 surface, proving that $|G| \leq 960$ and $G$ is isomorphic to a subgroup of the Mathieu group $M_{23}$. Furthermore if $|G| = 960$ then $G$ is isomorphic to the Mathieu group $M_{20}$. 
Kondō proved in \cite{Kondo} that the order of a finite group acting on a K3 surface is bounded by 3840 and it is only achieved when $X$ is the Kummer surface $\mathrm{Km}(E_i \times E_i)$, where $E_i$ is the usual notation for the elliptic curve $\C/(\Z\oplus i\Z)$. 
Bonnafé and Sarti in \cite{M20extension} studied K3 surfaces with finite maximal symplectic group $G_0=M_{20}$ which also admit a non-symplectic automorphism and found that there are three such K3 surfaces. Independently Brandhorst and Hashimoto in \cite{BrandhorstHashimoto} classified all finite groups acting faithfully on K3 surfaces with maximal symplectic part. 

In dimension bigger than two the lattice $H^2(X,\Z)$ is not unimodular, so the ideas have to be adapted. Beauville in \cite{BeauRemarks} generalized several results of Nikulin and Mongardi added new results in his Ph.D. thesis \cite{Mon-thesis}. In particular, these works give a pathway to classify the symplectic groups in the K3$^{[n]}$-type case. This was later applied by Höhn and Mason in \cite{HM} where they classify all the symplectic groups acting on IHS manifolds of K3$^{[2]}$-type. They found that there are 15 maximal symplectic such groups, and among them the group with the biggest order is $\Z_3^4 : \mathcal A_6$.

In the present article we generalize the results of Bonnaf\'e and Sarti in \cite{M20extension} to the case of K3$^{[2]}$-type IHS manifolds. We study the possible groups $G$ acting faithfully on an IHS manifold of K3$^{[2]}$-type such that the symplectic part $G_0$ is the group $\Z_3^4 : \mathcal A_6$. 
In other words, we consider the IHS manifolds of K3$^{[2]}$-type with a symplectic action of $\Z_3^4 : \mathcal A_6$ and such that they also admit a non-symplectic automorphism. We prove the following theorem.

\begin{thmINTRO}\label{theo1-intro}
   Let $X$ be an IHS manifold of K3$^{[2]}$-type that admits a symplectic action of $G_0=\Z_3^4:\mathcal A_6$ and such that the group $G_0$ can be extended to a group $G$ through a non-symplectic automorphism of order $m$. Then $m\in\{2,3,6\}$ and the transcendental lattice $T_X$ and the polarization $L$ are as in Propositions \ref{Order 2 cases} and \ref{Order 3 cases}.
\end{thmINTRO}

In Propositions \ref{Order 2 cases}, \ref{prop-Fermat}, \ref{Order 3 cases} and \ref{teo_final} and in Table \ref{OtherCases} we show more details and the characterization of these IHS manifolds. 

In particular as an analogous of the work of Kondō in \cite{Kondo} we are able to prove that there exists an IHS manifold of K3$^{[2]}$-type with the biggest automorphism group.

\begin{thmINTRO} \label{theo2-intro}
Let $X$ be an IHS manifold of K3$^{[2]}$-type with an action of $G$ such that
\begin{equation}
\xymatrix{
  1 \ar[r] & \Z_3^4 : \mathcal A_6 \ar[r] & G  \ar[r]^-{\alpha} & \mu_6  \ar[r] & 1 
},
\end{equation}
which would mean that $G$ has the biggest possible order $174960$. Then, we characterize $X$ by giving the transcendental lattice $T_X$, the polarization $L$ and its divisibility ${\rm div}(L)$  in Proposition \ref{teo_final}.
Moreover, the Fano variety of lines $F(Y)$ of the Fermat cubic fourfold \[ Y= \{ (x_0:\dots:x_5) : x_0^3 + \dots + x_5^3 = 0 \}\subset \PP^5 \]
is an example of IHS manifold of K3$^{[2]}$-type admitting the action of this $G$.
\end{thmINTRO}

\noindent We suspect that this IHS is the unique with the biggest automorphism group, see Subsection \ref{TowardsUnicity}.
A part of these results are also proven independently in a paper of Wawak \cite{W}, where he computes all the finite groups acting faithfully on IHS manifolds of K3$^{[2]}$-type with maximal symplectic part. 

The paper is organized as follows: in Section \ref{sec-prelim} we recall some preliminaries and fix notations about lattices and IHS manifolds, while in Section \ref{sec-groups} we study the action of automorphisms groups on IHS manifolds and present the example of the Fano variety of lines on a cubic fourfold. In Section \ref{sec-ext} we study the groups acting on an IHS manifolds of K3$^{[2]}$-type such that the symplectic part is $\Z_3^4 : \mathcal A_6$ and we prove the two main Theorems \ref{theo1-intro} and \ref{theo2-intro}.

\section*{Acknowledgments}
We are very grateful to Alessandra Sarti for suggesting the problem and for useful conversations. We are also grateful to Chiara Camere for sharing her insights with us. This work was started when P.Q.M. visited the Universit\'e de Poitiers; authors have been partially supported by Programa de cooperaci\'on cient\'ifica ECOS-ANID C19E06 and Math AmSud-ANID 21 Math 02.
P.C. and R.D. have been partially supported by Fondecyt Iniciaci\'on en la Investigaci\'on N.11190428 and P.C. has been partially supported by Fondecyt Regular N.1200608. P.Q.M. has been supported by ANID Becas-Doctorado Nacional N.21191367.

\section{Preliminaries and notation} \label{sec-prelim}
In this section we collect some background and known results about lattices and IHS manifolds. The standard references for lattices is \cite{Nikulin} and for IHS manifolds are \cite{BeauCons} and \cite{Debarre}.

    \subsection{Lattices}
    A \emph{lattice} $\mathcal L$ is a free $\Z$-module equipped with a non-degenerate symmetric bilinear form
    \begin{equation*}
        (\cdot, \cdot)_{\mathcal L} : \mathcal L \times \mathcal L \to \Z.
    \end{equation*}
    We omit the subscript if there is no confusion and for $v \in \mathcal L$ we write $ v^2:= (v, v)_{\mathcal L }$. The \emph{rank} of the lattice $\mathcal L$, denoted by $\rank(\mathcal L)$, 
    is the dimension of the real vector space $\mathcal L \otimes_\Z \R$ and the $\emph{determinant}$ of $\mathcal L$, denoted by $\det(\mathcal L)$, is the determinant of the Gram matrix $\{(e_i,e_j)\}_{i,j}$ for any basis $\{ e_1, \dots, e_n\}$ of $\mathcal L$. 
    The lattice is said to be $\emph{non-degenerate}$ if $\det(\mathcal L) \neq 0$ and $\emph{unimodular}$ if $|\det(\mathcal L)| = 1$. A lattice is $\emph{even}$ if $v^2 \in 2\Z$ for all $v \in \mathcal L$, \emph{odd} otherwise.  The \emph{signature} of a non-degenerate lattice is the pair $\sign(\mathcal L) := (s_+, s_-)$, where $s_+$ and $s_-$ are the numbers of positive and negative eigenvalues of $\mathcal L \otimes_\Z \R$, respectively. A lattice is \emph{positive (resp. negative) definite} if $\sign(\mathcal L) = (r,0)$ (resp. $(0,r)$) where $r = \rank(\mathcal L)$, otherwise it is an \emph{indefinite} lattice. A sublattice $M \subset \mathcal L$ is called \emph{primitive} if $M/\mathcal L$ is torsion-free.

    Let $\mathcal L$ be a lattice and let $\mathcal L^\vee := \Hom_\Z(\mathcal L,\Z)$ be the dual of $\mathcal L$. It can also be described as
    \[ \mathcal L^\vee := \{ v \in\mathcal L \otimes \Q : (v,w) \in \Z \ \forall w \in\mathcal L \}. \]
    Since $\mathcal L$ is a sublattice of $\mathcal L^\vee$ of the same rank, the \emph{discriminant group} $A_{\mathcal L} := \mathcal L^\vee / \mathcal L$ is a finite abelian group of order $|\det(\mathcal L)|$. We denote by $\pi :\mathcal L^\vee \to A_{\mathcal L}$ the quotient map. If $\mathcal L$ is even, the bilinear form on $\mathcal L$ induces a quadratic form $q_{\mathcal L} : A_{\mathcal L} \to \Q/2\Z$. A subgroup $H$ of $A_{\mathcal L}$ is called \emph{isotropic} if $q_{\mathcal L}|_H = 0$. Given a vector $v\in \mathcal L$, the \emph{divisibility} of $v$ is ${\rm div}(v,\mathcal L):=\gcd\{ (v,w): w\in \mathcal L\}$.

    We denote by $U$ the unique even unimodular lattice of rank two and signature $(1,1)$ and by $E_8$ the even, positive-definite lattice associated to the Dynkin diagram with the same name. Given a lattice $\mathcal L$, two integers $r,s\in\Z$ and $q$ the quadratic form on $\mathcal{L}$ associated to the bilinear form, then $\langle r \rangle$ denote the rank one lattice with value $q(l) = r$ on a generator $l\in \mathcal L$. Moreover $\mathcal L(s)$ denote the lattice with the form $(\cdot,\cdot)_{\mathcal L(s)} = s(\cdot,\cdot)_{\mathcal L}$.

    Since we will encounter it later, we recall that the Leech lattice $\mathcal L_{24}$ is the unique (up to isometry) rank 24 lattice which is positive definite,  even, unimodular and such that it admits no roots, where a root is a vector of square 2.
    
    An isomorphism of lattices preserving the bilinear form is called an $\emph{isometry}$. We denote by $\OO(\mathcal L)$ and $\OO(A_{\mathcal L})$ the group of isometries of a lattice $\mathcal L$ and its discriminant group $A_{\mathcal L}$, respectively. There exists a natural homomorphism $\OO(\mathcal{L}) \to \OO(A_\mathcal{L})$.
    
    Let $T$ be a fixed even lattice. An even lattice $M$ is an \textit{overlattice} of $T$ if there is an embedding $T \hookrightarrow M$ such that $M/T$ is a finite abelian group. 
    
    \begin{lem}[{\cite[Proposition 1.4.1]{Nikulin}}]\label{Overlattices} Let $T$ be an even lattice.
    There is a natural one-to-one correspondence between the finite index even overlattices $T \subset M$ and the isotropic subgroups $H \subset A_T$.
    \end{lem}
    
    \begin{proof}
    First, there is a natural chain of embeddings:
    \[ T \hookrightarrow M \hookrightarrow M^\vee \hookrightarrow T^\vee \]    
    such that we can associate to $M$ the isotropic subgroup of $A_T$ given by
    \[ H := M/T \subset M^\vee/T \subset T^\vee/T = A_T.  \]

    Conversely, we can associate to an isotropic group $H \subset A_T$ the overlattice $\pi^{-1}(H)$ where $\pi$ is the quotient map from $T^{\vee}$ to $A_T$.
    \end{proof}
    
    The following result is an application of \cite[Subsection 4]{Nikulin}.
    
    \begin{lem}\label{ExtendOver}
        Let $M$ be an overlattice of $T$. Then an element $\varphi \in \OO(T)$ can be extended to $M$ if and only if $\bar{\varphi}(H) = H$, where $H$ is the  isotropic subgroup of $A_T$ corresponding to $M$ and $\bar{\varphi} \in \OO(A_T)$ is the image of $\varphi \in \OO(T)$ under the projection $\pi$.
    \end{lem}
    In the case $T \subset M$ is a primitive sublattice, we can consider its orthogonal lattice $S = T^{\perp_M}$. Then $M$ is an overlattice of $S \oplus T$. We denote by
    \[ H_M := M/(S \oplus T) \subset A_{S \oplus T} \simeq A_S \oplus A_T \]
    the isotropic subgroup of $A_S \oplus A_T$ corresponding to $M$. We denote by $p_T : A_{T \oplus S} \to A_T$ and $p_S : A_{T \oplus S} \to A_S$ the projections and we define $M_T = p_T(H_M)$ and $M_S = 
    p_S(H_M)$. Since $T$ is primitive in $M$, the restrictions of the projections to $H_M$ are isomorphisms. We define the map $\gamma : p_S \circ (p_T^{-1})|_{M_T} : M_T \to M_{S}$. The map $\gamma$ is called the \emph{gluing morphism} and it is an anti-isometry, i.e. $q_T(x) = -q_S(\gamma(x))$ for every $x \in M_T$.

    \begin{lem}[{\cite[Corollary 1.5.2]{Nikulin}}]\label{ExtendComp} 
        Let $M$ be an even lattice with a primitive sublattice $T \subset M$ and orthogonal lattice $S$. For every isometry $\varphi \in \OO(T)$, there exists an isometry $\phi \in \OO(M)$ such that $\phi|_{T} = \varphi$ if and only if there exist an isometry $\psi \in \OO(S)$ such that $\bar{\psi} \circ \gamma = \gamma \circ \bar{\varphi}$.
    \end{lem}

\begin{definition}
Let $\mathcal L$ be a lattice and $G \subset \OO(\mathcal L)$. We define
\begin{align*}
    S^G(\mathcal L) &:= \{ x \in\mathcal L : g(x) = x \ \forall g \in G \}, \\
    S_G(\mathcal L) &:= {S^{G}(\mathcal L)}^{\perp_{\mathcal L}},
    \end{align*}
    the invariant and co-invariant lattices, respectively.
\end{definition}

\begin{rem}
    The invariant lattice $S^G(\mathcal L)$ and the co-invariant lattice $S_G(\mathcal L)$ are both primitive sublattices of $\mathcal{L}$.
\end{rem}

The following useful result is well known: 
    \begin{lem}\label{LemmaForIndex} Let $\mathcal L$ be a lattice and $G \subset \mathcal O(\mathcal L)$. Then the following hold:
    \begin{enumerate}
        \item $S^G(\mathcal L)$ contains $\sum_{g \in G} g(v)$ for all $v \in\mathcal  L$.
        \item $S_G(\mathcal L)$ contains $v - g(v)$ for all $v \in \mathcal L$ and $g \in G$.
        \item $\mathcal L/(S^G(\mathcal L) \oplus S_G(\mathcal L))$ is of $|G|$-torsion.
    \end{enumerate}
\end{lem}
\begin{proof}
    The proof of $(1)$ is obvious. For $(2)$, let $l \in S^G(\mathcal L)$, then for $v \in \mathcal L$ and $g \in G$ we have:
        \[ (l,v) = (g(l), g(v)) = (l, g(v)) \implies (l, v-g(v)) = 0.\]
    Thus, $v - g(v) \in S^G(\mathcal L)^\perp = S_G(\mathcal L)$.

    For $(3)$, let $l \in\mathcal  L$. Then
        \[ |G|l = \sum_{g \in G} g(l) + \sum_{g \in G} (l - g(l))\]
    and the first term lies in $S^G(\mathcal L)$ and the second in $S_G(\mathcal L)$. 
\end{proof}

    \subsection{Irreducible Holomorphic Symplectic (IHS) Manifolds.} A complex compact Kähler smooth manifold $X$ is called \emph{Irreducible Holomorphic Symplectic} (IHS) manifold
    if $X$ is simply connected and  $H^0(X,\Omega^2_X) = \C\omega_X$, where $\omega_X$ is a nowhere degenerate holomorphic 2-form on $X$. 
    The 2-dimensional IHS manifolds are K3 surfaces while for dimensions greater than two there are 4 known deformation families (see \cite{Debarre}): the IHS manifolds which are deformation equivalent to the Hilbert scheme of $n$ points on a K3 surface (K3$^{[n]}$-type), to the generalized Kummer surfaces, or to O'Grady's examples in dimension 6 or 10. 
    
    In this paper we study IHS manifolds of K3$^{[2]}$-type, i.e. manifolds which are deformation equivalent to the Hilbert scheme of $2$ points on a K3 surface. The construction for any $n$ is the following: let $S$ be a K3 surface; the Hilbert-Douady space $S^{[n]}$ parametrizes the zero-dimensional subschemes $(Z,\OO_Z)$ of the surface $S$ of length $n$ (i.e. $\dim_\C(\OO_Z)=n$). By \cite{BeauCons} it is a hyperkähler manifold of dimension $2n$. 
    
    Let $S^{(n)} := S^n/\mathcal S_n$, where $\mathcal S_n$ is the symmetric group on $n$ letters. The \emph{Hilbert-Chow} morphism, defined naturally as
    \begin{align*}
        s: S^{[n]} &\longrightarrow S^{(n)} \\
        Z &\longmapsto \sum_{p\in S} l(\OO_{Z,p}) p,
    \end{align*}
    is a desingularization of $S^{(n)}$, where $l(\OO_{Z,p})$ is the length of $\OO_{Z,p}$. 
    The case $n=2$ was studied by Fujiki in \cite{Fujiki} and its geometric description is particularly simple to work with. In this case, $S^{[2]} \to S^{(2)}$ is the blow-up of the symmetric square $S^{(2)}= (S\times S)/\mathcal S_2$ along the diagonal.

    If $X$ is an IHS manifold, the second cohomology group $H^2(X,\Z)$ equipped with the Beauville-Bogomolov-Fujiki (BBF) form has a lattice structure which is even, non-degenerate with signature $(3,b_2-3)$. For $X$ of K3$^{[n]}$-type (with $n \geq 2$), $H^2(X,\Z)$ has rank 23 and it is isometric to 
    \begin{equation}\label{Lattice K3^2}
        \Lambda_n := U^{\oplus 3} \oplus E_8(-1)^{\oplus 2} \oplus \langle -2(n-1) \rangle
    \end{equation}
    The N\'eron-Severi group of $X$ is defined as
    \[ NS(X) := H^{1,1}(X)_\R \cap H^2(X,\Z).\]
    We can see it as a lattice by considering the restriction of the BBF form to $H^2(X,\Z)$. It is even, non-degenerate and has signature $(1,\rho(X)-1)$ if $X$ is projective, where $\rho(X) := \rank \ NS(X)$ is called the \emph{Picard number} of $X$. Its orthogonal lattice $T_X := NS(X)^{\perp_{H^2(X,\Z)}}$ is called the \emph{transcendental lattice} of $X$. 
    
    A \emph{marking} on $X$ is an isometry $\eta: H^2(X,\Z) \to \Lambda$ of lattices. A \emph{marked IHS manifold} is a pair $(X,\eta)$ where $X$ is an IHS manifold together with an isometry $\eta: H^2(X,\Z) \to \Lambda$ on $X$. Two marked IHS manifolds $(X_1,\eta_1$) and $(X_2,\eta_2)$ are isomorphic if there exists an isomorphism $f: X_1 \to X_2$ such that $\eta_2 = \eta_1 \circ f^*$. 
    There is a coarse moduli space $\mathcal{M}_\Lambda$ parametrizing isomorphism classes of marked IHS manifolds. We call  \emph{Period domain} the set
    \[ \Omega_\Lambda := \{ [x] \in \PP(\Lambda \otimes \C) : x^2 = 0, \ (x, \bar x) > 0 \} \]
    The $\emph{period}$ of the marked pair $(X,\eta)$ is the point $\eta(H^{2,0}(X)) \in \PP(\Lambda \otimes \C)$ and we can check it lies on the Period domain.

The two following important results for IHS manifolds are used in what follows.
    \begin{thm}[Local Torelli Theorem] The period map
    \[
        \mathcal{P} : \mathcal{M}_\Lambda \longrightarrow \Omega_\Lambda,\quad 
        (X,\eta) \longmapsto \eta(H^{2,0}(X))
    \]
    is a local isomorphism. 
    \end{thm}
    Let $\mathcal{M}^0_\Lambda$ be a fixed connected component of $\mathcal{M}_\Lambda$. We consider the restriction of $\mathcal{P}$ to  $\mathcal{M}^0_\Lambda$:

    \[ P_0 := P|_{\mathcal{M}^0_\Lambda} \to \Omega_\Lambda \]
    \begin{thm}[Global Torelli Theorem] \label{surj} The map $P_0$ is surjective. Moreover, for every $x \in \Omega_\Lambda$, the fibre $P_0^{-1}(x)$ consists of pairwise birational manifolds.
    \end{thm}
    
An important example for the rest of the paper is the Fano variety of lines of a cubic fourfold, whose definition is the following: let $Y \subset \PP^5$ be a smooth cubic fourfold. The \emph{Fano variety of lines} on $Y$ is defined as
\[ F(Y) := \{ [\ell] \in \Gr(1,5) \ :\  \ell \subset Y \}. \]
Fano varieties of lines on cubic fourfolds were first studied by Beauville and Donagi in \cite{FanoVariety}, where the authors proved that they are IHS manifolds of K3$^{[2]}$-type.  
 Moreover, $F(Y)$ comes equipped with a polarization $h$ which is the restriction of the Plücker polarization of $\Gr(1,5)$ to $F(Y)$. Furthermore we have that $h^2 = 6$ and ${\rm div(h)} = 2$.

\section{Groups acting on IHS manifolds} \label{sec-groups}
Let $g \in \Aut(X)$ be an automorphism of finite order. If $g^*(\omega_X) = \omega_X$, then $g$ is called \emph{symplectic}. Otherwise, $g$ is called \emph{non-symplectic} and $g^*$ acts on $\omega_X$ as the multiplication by a scalar $\alpha(g)\in\C^*$. For a finite group $G \subset \Aut(X)$ we have the following exact sequence:
\begin{equation}\label{sequence}
\xymatrix{
  1 \ar[r] & G_0:=\ker\alpha \ar[r] & G  \ar[r]^-{\alpha} & \mu_m  \ar[r] & 1 
}.
\end{equation}
    The group $G_0$ is the \emph{symplectic} part of $G$, i.e. those automorphisms which act trivially on $\w_X$, and we call $\mu_m$ the non-symplectic part. 
Since automorphisms of $G$ have finite order, the group $\mu_m$ is cyclic.
    \begin{lem}[\cite{BeauRemarks}]\label{ActionTX} Let $X$ be an IHS manifold and let $G \subset \Aut(X)$ be a finite group. Then:
        \begin{enumerate}
            \item $g \in G$ acts trivially on $T_X$ if and only if $g \in G_0$.
            \item The representation of $\mu_m$ on $T_X \otimes \Q$ splits as the direct sum of irreducible representations of the cyclic group $\mu_m$ having maximal rank. 
        \end{enumerate}
    \end{lem}
This lemma implies that for an IHS manifold $X$ of K3$^{[n]}$-type we have a bound for $m$:
    \begin{equation}\label{BoundM}
       \mathcal T(m) \leq 23-\rho(X),
    \end{equation}
    where $\mathcal T$ is the Euler totient function.

Let $X$ be an IHS manifold of K3$^{[2]}$-type and let $G \subset \Aut(X)$ be a finite group acting on $X$. Let $(\eta,\Lambda)$ be a marking of $X$. By \cite{Debarre}
there is an injective map
        \[ \eta : \Aut(X) \to \mathcal O(\Lambda). \]
    One can identify $G$ with its image on $\mathcal O(\Lambda)$ and denote the invariant and co-invariant lattices by the action of $G$ respectively by $S^G(X) := S^G(\Lambda)$ and $S_G(X) := S_G(\Lambda)$.

    By Mongardi \cite{Mongardi1}, we have the following characterization of the co-invariant lattice when the action is by the symplectic part $G_0$ of $G$.
    \begin{lem}\cite[Lemma 2.10]{Mongardi1}\label{InvLattice} Let $X$ be an IHS manifold of K3$^{[n]}$-type and let $G_0$ be the symplectic part of a finite group $G$ acting on $X$. Then
    \begin{itemize}
        \item $S_{G_0}(X)$ is non-degenerate and negative definite.
        \item $S_{G_0}(X)$ contains no element with square -2.
        \item $T(X) \subset S^{G_0}(X)$ and $S_{G_0}(X) \subset NS(X)$.
        \item $G_0$ acts trivially on $A_{S_{G_0}(X)}$.
    \end{itemize}
    \end{lem}

Then, using the surjectivity of the period map of Theorem \ref{surj}, Mongardi in \cite{Mongardi1} proved the following theorem, which we state for the case of K3$^{[2]}$-type. Let $\mathcal{L}$ be a lattice isometric to $\Lambda_2=U^{\oplus 3} \oplus E_8(-1)^{\oplus 2} \oplus \langle -2 \rangle$ as in \eqref{Lattice K3^2}.
    \begin{thm}[{\cite[Theorem 7.2.2]{Mongardi1}}]\label{MongardiExistence}
        Let $G \subset \OO(\mathcal{L})$ be a finite group. Then $G$ is induced by a group of symplectic automorphisms for some marked IHS manifold $(X,\eta)$ of K3$^{[2]}$-type if, and only if, the following holds:
        \begin{itemize}
            \item $S_G(\mathcal{L})$ is negative-definite
            \item $S_G(\mathcal{L})$ contains no elements $v$ of norm $v^2 = -2$ or norm $v^2 = -10$ and $v / 2 \in \mathcal L^\vee$.
    \end{itemize}
\end{thm}

\section{K3$^{[2]}$-type manifolds with a symplectic action of $\Z_3^4 : \mathcal A_6$} \label{sec-ext}

In a recent work \cite{HM}, Höhn and Mason classify all symplectic groups that can act on IHS manifolds of K3$^{[2]}$-type.

\begin{thm}[{\cite{HM}}]\label{SymClass} Let $X$ be an IHS manifold of K3$^{[2]}$-type and let $G \subset \Aut(X)$ be a finite symplectic group. Then $G$ is isomorphic to one of the following:
\begin{enumerate}
    \item[(a)] A subgroup of $M_{23}$ with at least four orbits in its natural action on 24 elements. 
    \item[(b)] A subgroup of one of two subgroups $(\Z_3^{1+4}: \Z_2) \times \Z_2^2$ and $\Z_3^4:\mathcal A_6$ of the Conway group  $Co_0$ associated to a $\mathcal S$-lattice in $\mathcal L_{24}$.
\end{enumerate}
where a $\mathcal S$-lattice is a sublattice of the Leech lattice $M \subset \mathcal L_{24}$ on which all 
elements are congruent modulo $2\mathcal L_{24}$ to an element of $M$ of norm 0, -4 or -6. 
\end{thm}

For each case the authors in \cite{HM} compute the possible invariant and co-invariant lattices. In particular there are 13 isomorphism classes of subgroups of type $(a)$ and 2 subgroups of type $(b)$ that are maximal and for each maximal case $\rk S_{G}(X) = 20$.

Let $G$ be a group such that its symplectic part $G_0$ is maximal and such that $G/G_0$ is not trivial, i.e. in the exact sequence \eqref{sequence} one has $m>1$. The following Lemma establishes which values of $m$ are admissible in \eqref{sequence}. 

\begin{lem}\label{ms}
    In the previous setting, $m \in\{2,3,4,6\}$.
\end{lem}
\begin{proof}
By Lemma \ref{InvLattice} one has $ S_{G_0}(X) \subset NS(X)$ and since $m\neq 1$, then $NS(X)$ contains an ample class. Moreover, $S_{G_0}(X)$ is negative-definite and thus
\[ \rk  NS(X) \geq \rk  S_{G_0}(X) + 1.\]
Since maximal symplectic groups have $\rk S_{G_0}(X) = 20$ by \cite{HM}, then $\rho(X) \geq 21$ and by \eqref{BoundM} one has $ \mathcal T(m) \leq 2 $.
As a consequence, in the exact sequence \eqref{sequence} possible $m$'s are $m =2,3,4,6$.
\end{proof}
In the list of \cite{HM} the group with biggest order is $G_0:= \Z_3^4 : \mathcal A_6$, with $|G_0|= 29160$. Thus, let $G \subset \Aut(X)$ be a group whose symplectic part $G_0$ is maximal. Since we know that in the exact sequence \eqref{sequence} the bigger $m$ is 6 by Lemma \ref{ms}, then the order of $G$ is bounded by 
\begin{equation}
    \label{bound}
 |G| \leq 29160 \times 6 = 174960.\end{equation}
In fact we will show that there exists such a group $G$ for which \eqref{bound} is an equality.
\begin{rem}
We can observe that if the group $G_0$ is not maximal, then $m$ can be bigger than 6 but according to the list in \cite{HM}, the order of $G$ can not be bigger than the bound of \eqref{bound}.
\end{rem}

From now on, let $G_0=\Z_3^4 : \mathcal A_6$, with $|G_0|= 29160$. By \cite{HM}, the invariant lattice $S^{G_0}(X)$ is isometric to
\[ S^{G_0}(X) = \begin{pmatrix} 6 & 3 & 0 \\ 3 & 6 & 0 \\ 0 & 0 & 6 \end{pmatrix}.\]

We want to study the possible extensions $G$ of $G_0$, i.e. the groups acting on an IHS manifold such that their symplectic part is the group $G_0= \Z_3^4 : \mathcal A_6$. 

If $X=F(Y)$ is the Fano variety of lines of a cubic fourfold $Y$, an automorphism on $Y$ naturally induces an automorphism on $X=F(Y)$. Moreover, the following holds (see \cite[Lemma 1.2, Corollary 1.3]{Fu}).

\begin{lem}\label{AutInjective} 
     An automorphism $\psi$ of $F(Y)$ is induced by an automorphism of $Y$ if and only if $\psi^*(h) = h$, where $h$ is the polarization of $F(Y)$.
    The natural morphism 
    \[ \Aut(Y) \to \Aut(F(Y)) \]
    is injective and its image is denoted by $\Aut_h(F(Y))$ and consists of automorphisms of $F(Y)$ which fix the polarization $h$.
\end{lem}

In particular, let $Y$ be the Fermat cubic fourfold:
    \[ Y= \{ (x_0:\dots:x_5) \in \PP^5 : x_0^3 + \dots + x_5^3 = 0 \}\subset \PP^5. \]

It is proven in \cite[Theorem 1.8]{CubicFourfolds} that $Y$ is the only cubic fourfold which admits the biggest possible automorphism group $\Aut(X) = \Z_3^5 : \mathcal S_6$.
Using Lemma \ref{AutInjective}, we can consider the Fano variety $F(Y)$ and show that $F(Y)$ admits the action of $\Z_3^5 : \mathcal S_6$. We observe that the order of $\Z_3^5 : \mathcal S_6$ is 174960, thus this group fits 
in the sequence
\begin{equation}
\xymatrix{
  1 \ar[r] & \Z_3^4 : \mathcal A_6 \ar[r] & G  \ar[r]^-{\alpha} & \mu_6  \ar[r] & 1 
}.
\end{equation}

Mongardi in \cite[Example 7.4.2]{Mon-thesis} proved that \[T_{F(Y)} \simeq \begin{pmatrix}
    6&3\\3&6
\end{pmatrix} \mbox{ and }\  S^{G_0}(F(Y)) \simeq \langle 6 \rangle \oplus \begin{pmatrix}
    6&3\\3&6
\end{pmatrix}.\]

\begin{rem}
 \label{unique} Let $Y$ be a cubic fourfold and let $G = \Z_3^5 : \mathcal S_6$ acting on the Fano variety of lines $F(Y)$ of $Y$ fixing the polarization $h$. Since the morphism
\[ \Aut(Y) \to \Aut_h(F(Y)) \]
is an isomorphism, we can consider the pre-image of $G \subset \Aut_h(F(Y))$ which acts on $Y$. 
Since the Fermat cubic fourfold is the only cubic fourfold which admits an action of $G$, $Y$ has to be the Fermat cubic fourfold, and thus $F(Y)$ is the Fano variety of lines of the Fermat cubic fourfold.
\end{rem}

\subsection{Extensions of $G_0 = \Z_3^4:\mathcal A_6$ by a non-symplectic automorphism}\label{Section-Extensions}
Let $X$ be an IHS manifold of K3$^{[2]}$-type and let $G \subset \Aut(X)$ such that $G_0 = \Z_3^4:\mathcal A_6$ and $X$ admits a non-symplectic automorphism of order $m$. We have the following exact sequence
\begin{equation}
\xymatrix{
  1 \ar[r] &G_0 \ar[r] & G  \ar[r] & \mu_m  \ar[r] & 1 
}.
\end{equation}
where $G/G_0 = \mu_m$ with $m \in \{2,3,4,6\}$. As we observed, $\rho(X) = 21$, $\rk T_X = 2$ and $T_X \subset S^{G_0}(X)$. 

Let $\{e,f,h\}$ be the basis such that
\[ S^{G_0}(X) = \begin{pmatrix} 6 & 3 & 0 \\ 3 & 6 & 0 \\ 0 & 0 & 6 \end{pmatrix}.\]
By \cite{HM}, we know that the isometry group of $S^{G_0}(X)$ is of order 24 and we have the following:

\begin{lem}
\label{Isometries} The isometry group $\mathcal O(S^{G_0}(X))$ of $S^{G_0}(X)$ does not contain any element of order 4.\end{lem}
\begin{proof}
Denote by $\rho_1$, $\rho_2$ and $\rho_3$ 
the isometries of $S^{G_0}(X)$ defined by the following matrices:
\[  \rho_1 = \begin{pmatrix} 1 & 0 & 0 \\ 0 & 1 & 0 \\ 0 & 0 & -1 \end{pmatrix}, \ \ \,  \rho_2 = \begin{pmatrix} -1 & 1 & 0 \\ 0 & 1 & 0 \\ 0 & 0 & 1 \end{pmatrix} \ \ \ \textnormal{and} \ \ \  \rho_3 = \begin{pmatrix} 0 & -1 & 0 \\ 1 & -1 & 0 \\ 0 & 0 & -1 \end{pmatrix}.\]
A direct computation by Sage shows that the isometry group $\mathcal O(S^{G_0}(X))$ is spanned by $-Id$, $\rho_1$, $\rho_2$ and $\rho_3$ and it 
does not contain any element of order 4.\end{proof}

This allows to show that in the previous setting, $X$ does not admit a non-symplectic automorphism of order 4 and thus the value of $m$ in \eqref{sequence} can not be 4.

\begin{prop}
Let $X$ be an IHS manifold of K3$^{[2]}$-type and let $G \subset \Aut(X)$ such that $G_0 = \Z_3^4:\mathcal A_6$ and $X$ admits a non-symplectic automorphism of order $m$. Then $m \in \{2,3,6\}$.
\end{prop}

\begin{proof}
    By Lemma \ref{ms} $m \in \{2,3,4,6\}$. To discard the case $m=4$, we notice that the group $G/G_0 \simeq \mu_m$ acts on $S^{G_0}(X)$ faithfully, but $\OO(S^{G_0}(X))$ does not contain an element of order 4 by Lemma \ref{Isometries}.
\end{proof}

\begin{rem}  \label{ord6}
 We will study the cases $m=2$ and $m=3$ and show that the same IHS manifold $X$ admits the action of both non-symplectic automorphisms. Since 2 and 3 are coprimes, this would imply that $X$ also admits a non-symplectic automorphism of order 6.
\end{rem}

\begin{rem}\label{ValuesLattice} Let $\{e,f,h\}$ be the standard basis of $S^{G_0}(X)$ and let $L \in S^{G_0}(X)$. Then if $L = \lambda e + \mu f + \delta h$, with $\lambda,\mu,\nu\in\Z$, $L^2$ has the following form:
\[L^2= \left(\begin{array}{ccc}
\lambda & \mu & \delta \\
\end{array}\right)\,
\left(\begin{array}{ccc}
6 & 3 & 0 \\
3 & 6 & 0 \\
0 & 0 & 6 \\
\end{array}\right)\,
\left(\begin{array}{c}
\lambda \\ \mu \\ \delta \\
\end{array}\right)=6\left(\lambda^2+\lambda\mu+\mu^2+\delta^2\right). \]
Thus $L^2\in 6\mathbb{Z}$. Since $\lambda^2 + \lambda\mu + \mu^2 \geq 0$ and thus $L^2 > 0$, we can assume that $L^2=6n$ with $n$ a positive integer.
In the same way $(L, M) \in 3\Z$ for $L,M \in S^{G_0}(X)$.
\end{rem}

From now on we will assume that $G = \langle i\rangle G_0$, with $i$ a non-symplectic automorphism normalizing $G_0$ and such that $i^m \in G_0$, with $m \in \{2,3\}$. The group $G/G_0 = \langle i\rangle$ acts on $S^{G_0}(X)$ fixing a polarization $L$ with $L^2 = 6n$ and acting as an order $m$ isometry on $T_X$. We observe that the polarization $L$ is primitive and that $T_X = L^\perp \cap S^{G_0}(X)$. Furthermore, $L \oplus T_X$ is a sublattice of $S^{G_0}(X)$ of the same rank (and finite index).

\begin{prop}\label{Index2} Let $L$ be an element in $S^{G}(X)$ and let $L^2 = 6n$. Then, $L \oplus T_X = \langle 6n \rangle \oplus T_X$ is a sublattice of $S^{G_0}(X)$ of index $m$ except in the case that $n = 1$, $L = \pm h$ in the standard basis, 
and 
\[ T_X = \begin{pmatrix} 6 & 3 \\ 3 & 6 \end{pmatrix},\]
for which we have that $S^{G_0}(X) = L \oplus T_X = \langle 6 \rangle \oplus \begin{pmatrix}
    6&3\\3&6
\end{pmatrix}$.
\end{prop}
\begin{proof}
Let $L$ be an element in $S^{G}(X)$ with $L^2 = 6n$. We want to find for which values of $n$ we have that $L \oplus T_X$ is equal to $S^{G_0}(X)$ and for which values of $n$ it is a sublattice of $S^{G_0}(X)$. By Remark \ref{ValuesLattice}, $L = \langle6n\rangle$ for some $n \geq 0$ and 
\[ T_X = \begin{pmatrix} 6a & 3b  \\ 3b & 6c \end{pmatrix},\]
with $-a < b \leq a \leq c$ and $b \geq 0$ if $a=c$. In particular, $27b^2 \leq \disc(T_X) = 36ac-9b^2$. Thus, if $ S^{G_0}(X) = L \oplus T_X$, then
\[ 162 = \disc(S^{G_0}(X)) = \disc(L)\disc(T_X) = 6n(36ac-9b^2) = 54n(4ac-b^2).\] This implies $3 = n(4ac-b^2) $
and thus the only possible values of $n$ are $1$ or $3$. 
The case $n = 3$ is not admissible since this would imply $4ac-b^2=1$ but there is no $b\in\Z$ such that $b^2\equiv 3 \mod 4$. 
If $n=1$, then $\det(T_X) = 27$ and the unique possibility is that $a=b=c=1$ and $L = \pm h$. We conclude that for all other $n$ and $T_X$ the lattice $\langle 6n \rangle \oplus T_X$ is a sublattice of $S^{G_0}(X)$. 

To compute the index we follow \cite[Lemma 2.9]{M20extension}. We are assuming that $L \oplus T_X$ is different from $S^{G_0}(X)$. Since $L$ is primitive in $S^{G_0}(X)$, 
the projection
\[ S^{G_0}(X)/(L \oplus T_X) \to A_L \]
is an embedding. This shows that $S^{G_0}(X)(L \oplus T_X)$ is cyclic, and by Lemma \ref{LemmaForIndex} it is $m$-torsion with $m$ prime. Thus, $L \oplus T_X$ has index $m$ in $S^{G_0}(X)$.
\end{proof}

Thanks to the following result we can characterize the form of the matrix of the transcendental lattice depending if it admits an order 3 isometry.
\begin{lem}[{\cite[Theorem 51a]{QuadraticForms}}]\label{Admits} Let $T$ be a positive definite rank 2 even lattice. Then $T$ admits an isometry of order 3 if and only if there exists $a\in\Z_{>0}$ such that $T$ has the form
\[ T_X = \begin{pmatrix} 2a & a  \\ a & 2a \end{pmatrix},\]
\end{lem}

\subsubsection{Extensions by an order 2 non-symplectic automorphism}
We start by classifying the possible transcendental lattices of an IHS manifold of K3$^{[2]}$-type together with a polarization such that it admits the symplectic action of $G_0$ and a non-symplectic involution. In other words, in the exact sequence \eqref{sequence} the group $G$ is spanned by $G_0$ and $\mu_2$.

\begin{prop}\label{PossibleEmbedding} Assuming that $X$ admits a non-symplectic involution, then the only values of $n$ such that there is an embedding of $\langle 6n \rangle \oplus T_X$ in $S^{G_0}(X)$ are $n=1,3,4$ . 
\end{prop}

\begin{proof}
Since $S^{G_0} = \langle 6 \rangle \oplus \begin{pmatrix}
    6&3\\3&6
\end{pmatrix}$, the case $n = 1$ always happens. Let us consider the case that $L^2 = 6n$ and $L \oplus T_X$ is a sublattice of index 2 (by Proposition \ref{Index2}) in $S^{G_0}(X)$. Then, by \cite[Lemma I.2.1]{BarthHulek} we have that:
\[ 4 = [S^{G_0}(X):\langle6n\rangle\oplus T_X]^2 = \frac{\det(\langle6n\rangle \oplus T_X)}{\det(S^{G_0}(X))} = \frac{54n(4ac-b^2)}{162}. \] This implies $12 = n(4ac-b^2)$
and thus $n \in \{1,2,3,4,6,12\}$. Notice that the equation
\begin{equation}\label{Relation}
    \frac{12}{n} = 4ac-b^2 
\end{equation}
does not admit integers solution for $n = 2,\, 6$ and $12$. Thus $n \in \{1,3,4\}$. 
\end{proof}
 
Since the transcendental lattice $T_X$ has the form
\[ T_X = \begin{pmatrix} 6a & 3b  \\ 3b & 6c \end{pmatrix},\]
by Equation \eqref{Relation} one has
\begin{equation}\label{Det}
    \det(T_X) = 9(4ac-b^2) = 9\frac{12}{n}.
\end{equation}

\begin{prop}\label{Order 2 cases} Let $X$ be an IHS manifold of K3$^{[2]}$-type. Assume that $G_0 = \Z_3^4 : \mathcal A_6$ acts symplectically and faithfully on $X$ and assume that $X$ also admits a purely non-symplectic automorphism $\sigma$ of order $m=2$ acting on it, normalizing $G_0$ and such that $\sigma^2 \in G_0$. Let $G = \langle \sigma \rangle G_0$. Then, up to embedding and isometry of the polarization, we have the following possibilities:
\begin{enumerate}
    \item $L = h$, $L^2 = 6$, ${\rm div}(L) = 2$ and $T_X =  \begin{pmatrix} 6 & 3 \\ 3 & 6 \end{pmatrix}$.
    \item $L = e$, $L^2 = 6$, ${\rm div}(L) = 1$ and $T_X =  \begin{pmatrix} 6 & 0 \\ 0 & 18 \end{pmatrix}$.
    \item $L = e-f$, $L^2 = 6$, ${\rm div}(L) = 1$ and $T_X =  \begin{pmatrix} 6 & 0 \\ 0 & 18 \end{pmatrix}$.
    \item $L = e+f$, $L^2 = 18$ and $T_X =  \begin{pmatrix} 6 & 0 \\ 0 & 6 \end{pmatrix}$.
    \item $L = 2e-f$, $L^2 = 18$ and $T_X =  \begin{pmatrix} 6 & 0 \\ 0 & 6 \end{pmatrix}$.
\end{enumerate}
where $\{e,f,h\}$ is the standard basis of the invariant lattice $S^{G_0}(X)$.
\end{prop}

\begin{proof}
Using Proposition \ref{PossibleEmbedding}, we study the possible embeddings of $L$ with $L^2 = 6n$, for $n = 1$, $3$ and $4$. For each case we look for the integers $\lambda,\mu,\delta$ such that $L=\lambda e+\mu f+\delta h$ satisfies
\begin{equation} \label{6n} 
L^2=6\left(\lambda^2+\lambda\mu+\mu^2+\delta^2\right)=6n\end{equation}
and we use $L$ in order to determine the transcendental lattice. 
\begin{table}
\begin{tabular}{|c|c|p{2.7cm}|p{4.3cm}|}
    \hline
    $L^2$ & Value of $\delta$ & Possibles $(\lambda,\mu)$ & Embedding (up to isometry) \\
    \hline\hline
    
    \multirow{3}{*}{$L^2=6$}
     & $\delta=\pm 1$ & $\{(0,0)\}$ & $L\mapsto h$  \\
     \cline{2-4}
     & \multirow{2}{*}{$\delta=0$}
     & $\{(\pm 1,0),(0,\pm 1)\}$ & $L\mapsto e$  \\
     \cline{3-4}
     & & $\{(1,-1),(-1,1)\}$ & $L\mapsto e-f$ \\
     \hline
     \multirow{2}{*}{$L^2=18$}
      & \multirow{2}{*}{$\delta=0$}
      & $\{(1,1),(-1,-1)\}$ & $L\mapsto e+f$ \\
      \cline{3-4}
      & & $\{(2,-1),(-2,1),$ $(1,-2),(-1,2)\}$ & $L\mapsto 2e-f$ \\
      \hline
      \multirow{5}{*}{$L^2=24$}
      & $\delta=\pm 2$ & $\{(0,0)\}$ & $L\mapsto 2h$  \\
      \cline{2-4}
      & \multirow{2}{*}{$\delta=\pm 1$}
      & $\{(1,1),(-1,-1)\}$ & $L\mapsto e+f+h$  \\
      \cline{3-4}
      & & $\{(2,-1),(-2,1),$ $(1,-2),(-1,2)\}$ & $L\mapsto 2e-f+h$  \\
      \cline{2-4}
      & \multirow{2}{*}{$\delta=0$}
      & $\{\pm 2, \pm 2 \}$ & $L\mapsto 2e$  \\
      \cline{3-4}
      & & $\{(2,-2),(-2,2)\}$ & $L\mapsto 2(e-f)$ \\
      \hline
\end{tabular}
\caption{Computation for $m=2$}
\label{table_orbits}
\end{table}
\begin{itemize}
    \item {\bf Case $L^2 = 6$.}
    The possible values of $\delta$ are $\delta \in \{-1,0,1\}$. By the description of $\mathcal O(S^{G_0}(X))$ given in the proof of Lemma \ref{Isometries} one gets that there are 3 possibles orbits, up to isometry, see Table \ref{table_orbits}.
    We study the three cases separately. 
   
   The first case is $(\lambda,\mu, \delta)=(0,0,1)$, i.e. the embedding $L = h$.
   Here we have the equality $S^{G_0} = \langle 6 \rangle \oplus T_X$ by Proposition \ref{Index2} and then we can write $T_X$ on the basis $\{e,f\}$ as
    $$T_X = \begin{pmatrix} 6 & 3 \\ 3 & 6 \end{pmatrix}.$$
     
    For the other two cases we have that $L \oplus T_X$ is a sublattice of index 2 by Proposition \ref{Index2} and we expect that $\det(T_X) = 108$ by Equation \eqref{Det}.  
    Let $(\lambda,\mu, \delta)$ be $(1,0,0)$, which corresponds to $L = e$. 
    We can compute the transcendental lattice as $T_X = L^\perp \cap S^{G_0}(X)$, obtaining that $T_X$ has generators $\{h, e-2f\}$ and intersection matrix equal to
    $$T_X = \begin{pmatrix} 6 & 0 \\ 0 & 18 \end{pmatrix}.$$
    
   The last case is $(\lambda,\mu, \delta)=(1,-1,0)$, i.e. $L = e-f$. Then the transcendental lattice has generators $\{h,e+f\}$ and intersection matrix
    $$T_X = \begin{pmatrix} 6 & 0 \\ 0 & 18 \end{pmatrix}.$$
    
    \item {\bf Case $L^2 = 18$.}
    First observe that $\lambda^2 + \lambda\mu + \mu^2$ is never equal to 2: if $x$ and $y$ are both odd integers, or if one is odd and the other is even, then $x^2+xy+y^2$ will be odd and so different from 2. Assume that $x$ and $y$ are both even and observe that they are both different from 0. 
    If $x$ and $y$ have the same sign, then $xy>0$ and so $x^2+xy+y^2\geq 3$. It remains to check when $x$ and $y$ have different sign. As the equation is symmetric, we can assume
$$x=2a~\mathrm{and}~y=-2b,~\mathrm{with}~a,b\in\mathbb{Z}_{>0}.$$
    Then $x^2+xy+y^2=4\left(a^2+b^2-ab\right)$ and since $a^2+b^2-ab=(a-b)^2+ab>0$, it is impossible to have $x^2+xy+y^2$ equal to 2.\\
    
    It follows that the only possible value of $\delta$ for \eqref{6n} to hold is $\delta = 0$. There are two orbits of $(\lambda,\mu,\delta)$ in $\mathcal O(S^{G_0}(X))$ (see Table \ref{table_orbits}) and in both cases we have that $L \oplus T_X$ is a sublattice of index 2 on $S^{G_0}(X)$ such that $\det(T_X) = 36$. 

    Let $L$ be $L = e+f$. Then we can compute the transcendental lattice as $T_X = L^\perp \cap S^{G_0}(X)$, obtaining the generators $\{e-f,h\}$ and intersection matrix
    $$T_X = \begin{pmatrix} 6 & 0 \\ 0 & 6 \end{pmatrix}.$$
    
    Now let $ L$ be $L = 2e-f$. Then the transcendental lattice has generators $\{f,h\}$ and intersection matrix
    $$T_X = \begin{pmatrix} 6 & 0 \\ 0 & 6 \end{pmatrix}.$$
    
    \item {\bf Case $L^2 = 24$.}
    In this last case, if $(\lambda,\mu,\delta)$ satisfy \eqref{6n} then $\delta \in \{0,\, \pm 1,\, \pm 2\}$ and we compute the possible orbits up to isometry (see Table \ref{table_orbits}). In all cases we expect $L \oplus T_X$ is a sublattice of index 2 on $S^{G_0}(X)$ such that $\det(T_X) = 27$ by Equation \eqref{Det}. 
    The only primitive cases are $L=e+f+h$ and $L=2e-f+h$ and in both cases one can compute that $$T_X = \begin{pmatrix} 6 & 0 \\ 0 & 72 \end{pmatrix}$$ whose determinant is not 27. Thus there is no possible case with $L^2=24$.
\end{itemize}
The divisibility of the polarization $L$ for each case is computed in Lemma \ref{LemmaDiv}.
\end{proof}

By Remark \ref{unique}, the only Fano variety of lines on a cubic fourfold with symplectic action of $G_0=\Z_3^4:\mathcal A_6$ is the Fano variety $F(Y)$ of lines on the Fermat cubic fourfold \[ Y= \{ (x_0:\dots:x_5) : x_0^3 + \dots + x_5^3 = 0 \}\subset \PP^5. \]
Thus this is also the unique Fano variety of cubic fourfold with symplectic action of $G_0$ and
admitting a non-symplectic involution with $L=h$. Thus the following holds.
\begin{prop}\label{prop-Fermat}
Let $X=F(Y)$ be the Fano variety of lines of a cubic fourfold such that $G_0=\Z_3^4 : \mathcal A_6$ acts faithfully and simplectically on $X$ and $X$ admits a non-symplectic involution. Then $Y$ is the Fermat cubic fourfold.
\end{prop}

\subsubsection{Extensions by an order 3 non-symplectic automorphism}
We classify now  possible transcendental lattices and polarizations of an IHS manifold $X$ of K3$^{[2]}$-type  such that $X$ admits the action of $G_0$ together with a non-symplectic automorphism of order 3.

\begin{prop}\label{Order 3 cases} Let $X$ be an IHS manifold of K3$^{[2]}$-type with symplectic and faithful action of $G_0$. If $X$ admits an order $3$ non-symplectic automorphism, then the only possibility is  $L = h$, with ${\rm div}(L) = 2$ and
\[ T_X =  \begin{pmatrix} 6 & 3 \\ 3 & 6 \end{pmatrix}.\]

Moreover, if $X=F(Y)$ is the Fano variety of lines of a cubic fourfold, then $Y$ is the Fermat cubic fourfold, $G_0=\Z_3^4 : \mathcal A_6$ acts faithfully and simplectically on $X$ and $X$ admits an order 3 non-symplectic automorphism. 
\end{prop}
\begin{proof}
By Remark \ref{ValuesLattice} and Lemma \ref{Admits} 
we have that if $T_X$ admits an order 3 automorphism then 
\[ T_X = \begin{pmatrix} 6a & 3a  \\ 3a & 6a \end{pmatrix},\quad a \in \Z_{> 0}.\]

Let $L = \lambda e + \mu f + \delta h,\lambda,\mu,\nu\in\Z$ with $L^2 = 6n$. 
For $n = 1$, $L^2 = 6$ the only possibility such that $T_X = a\cdot \begin{pmatrix}
    6&3\\3&6
\end{pmatrix}$ is $S^{G_0}(X) = \langle 6 \rangle \oplus \begin{pmatrix}
    6&3\\3&6
\end{pmatrix}$ and $L = \pm h$. 

If we assume that $n > 1$, then $L \oplus T_X$ is a sublattice of index 3 in $S^{G_0}(X)$ by Proposition \ref{Index2} and we have that:
\[ 9 = [S^{G_0}(X):\langle6n\rangle\oplus T_X]^2 = \frac{\det(\langle6n\rangle \oplus T_X)}{\det(S^{G_0}(X))} = \frac{6n(27a^2)}{162}. \]
This implies $9 = na^2$ and thus $n = 9$ and $T_X = \begin{pmatrix}
    6&3\\3&6
\end{pmatrix}$. 

In this case let $L^2 = 54$, i.e. by \eqref{6n}
\[\lambda^2+\lambda\mu+\mu^2+\delta^2=9.\]
We study each case separately according to the value of $\delta \in \{0, \pm 1, \pm 2, \pm 3\}$.
\begin{itemize}
    \item If $\delta = \pm 3$ then $\lambda = \mu = 0$, which implies $L = 3h$ and this is not a primitive case.
    \item If $\delta = \pm 2$ then $\lambda^2 + \lambda\mu + \mu^2 = 5$, which have no integers solutions.
    \item If $\delta = \pm 1$ then $\lambda^2 + \lambda\mu + \mu^2 = 8$, which have no integers solutions.
    \item If $\delta = 0$ then $\lambda^2 + \lambda\mu + \mu^2 = 9$ and $L$ is one of the following:
    \[\pm 3e,\ \pm 3f,\ \pm 3(e-f)\]
    but in each case the polarization is non-primitive. 
\end{itemize}

By Remark \ref{unique}, the only Fano variety of cubic fourfold with symplectic action of $G_0=\Z_3^4:\mathcal A_6$ is the Fano variety of lines of the Fermat cubic fourfold $Y\subset\PP^5$.
Thus this is also the unique Fano variety of cubic fourfold with symplectic action of $G_0$ 
admitting a non-symplectic automorphism of order 3 with $L=h$. The divisibility of the polarization $L$ is computed in Lemma \ref{LemmaDiv}.
\end{proof}

As we observed in Remark \ref{ord6}, studying the existence of a non-symplectic involution and a non symplectic automorphism of order 3 is enough to ensure the existence of a non-symplectic automorphism of order 6. Thus as a consequence of Propositions \ref{Order 2 cases}, \ref{prop-Fermat} and \ref{Order 3 cases},  i.e. the results for the cases $m=2$ and $m=3$, the following result is straightforward.

\begin{prop}\label{teo_final}
Let $X$ be an IHS manifold of K3$^{[2]}$-type with symplectic and faithful action of $G_0$.
If $X$ admits an order $6$ non-symplectic automorphism, which would means that $X$ admits the action of the biggest possible finite group $G$ with $|G|=29160\times 6=174960$, then the only possibility is $L = h$, ${\rm div}(L) = 2$ and
\[ T_X =  \begin{pmatrix} 6 & 3 \\ 3 & 6 \end{pmatrix}.\]
Moreover, the Fano variety of lines $F(Y)$ of the Fermat cubic fourfold \[ Y= \{ (x_0:\dots:x_5) : x_0^3 + \dots + x_5^3 = 0 \}\subset \PP^5. \]
is an example of this case, admitting the action of the biggest possible finite group $G\subset\Aut(X)$.
\end{prop}

\subsection{Towards Unicity}\label{TowardsUnicity}
We suspect that the Fano variety of lines of the Fermat cubic fourfold is the only IHS manifold of K3$^{[2]}$-type admitting the action of a group $G$ with the maximal possible order, that is $|G| = 174960$.

Let $X$ be an IHS manifold of K3$^{[2]}$-type. 
By \cite{Mon-thesis}, there is a unique action of $G_0$ on $H^2(X,\Z)$, and if we add the requirement that it admits an order 6 non-symplectic automorphism, then $T_X = \begin{pmatrix}
    6&3\\3&6
\end{pmatrix}$ and $NS(X) = \langle-2\rangle \oplus \begin{pmatrix}
    -6&-3\\-3&-6
\end{pmatrix} \oplus U \oplus E_8(-1)^{\oplus 2}$. 

The IHS manifold $X$ is polarized by an  ample class $L$ such that $L^2 = 6$ and $\rm{div}(L) = 2$. 
Let ${}^2\!{\mathcal{M}_{2n}^{(2)}}$ be the moduli space parametrizing IHS manifolds of K3$^{[2]}$-type with a polarization of square $2n$ and divisibility $2$. 
The general element of  ${}^2\!{\mathcal{M}_6^{(2)}}$
is a Fano variety of lines of cubic fourfolds (see \cite[Section 3.6]{Debarre}). 
Thus the natural question is: "\emph{Is there any other IHS manifold in the moduli space ${}^2\!{\mathcal{M}_6^{(2)}}$
which is not the Fano variety of lines of a cubic fourfold and such that it admits the action of the biggest possible group?}". 
If the answer is yes, the IHS manifold would be a special element in ${}^2\!{\mathcal{M}_6^{(2)}}$.

An $M$-polarized IHS manifold of K3$^{[2]}$-type is a tuple $(X,\nu, j)$ where $X$ is a projective IHS manifold of  K3$^{[2]}$-type with a marking $\nu: H^2(X,\Z) \to \Lambda_2$ and $j : NS(X) \hookrightarrow M$ is a primitive embedding of lattices such that $j^{-1}(M) \subset NS(X)$ and $\nu|_{NS(X)} = j$.

Let $N := M \cap \Lambda_2$ and set
\[ \Omega_M := \{ x \in \PP(N\otimes\C) : x^2=0, (x,\bar x) > 0 \}, \]
the period domain $\Omega_M$, which consist of a disjoint union of two components of dimension $21-\rho(X)$.

As in \cite[Section 5.2]{BCS}, one can construct a moduli space $\mathcal{K}_M$ of $M$-polarized IHS manifold of K3$^{[2]}$-type with the respective period domain $\Omega_M$ and Period map $\mathcal{P} : \mathcal{K}_M \to \Omega_M$. 

Let $X$ be the Fano variety of lines of the Fermat cubic fourfold and let $M$ be a sublattice of $\Lambda_2$ isometric to $NS(X)$. Since $\OO(M) \to \OO(q_M)$ is surjective, we can consider a marking $\nu: H^2(X,\Z) \to \Lambda_2$ such that $\nu^{-1}(M) \subset NS(X)$ is a primitive embedding. Since $\rho(X) = 21$, $\Omega_M$ consists of two periods that correspond to an IHS manifold whose N\'eron-Severi group $NS(X)$ is isometric to $M$. These two points come from the choice of an orientation in the positive cone of $\Lambda_2$ (see \cite[Corollary 9.10]{Markman}).

Every possible embedding of $NS(X)$ into $H^2(X,\Z)$ gives us a different family, but there is a unique embedding of $NS(X)$ with the orthogonal lattice $T_X$ required. 
By Theorem \ref{surj}, the fiber of $\mathcal P$ on each of these points consists of inseparable points 
which corresponds to IHS manifolds birational to $F(Y)$ by \cite[Theorem 2.2]{Markman}. 
Thus the question becomes: "\emph{Is there an IHS manifold $X$ of K3$^{[2]}$-type 
which is birational to $F(Y)$,
where $Y$
is the Fermat cubic fourfold, with the aforementioned $NS(X)$ and $ T_X$, the polarization $L$, and admitting the action of the biggest possible group?}".
We are currently working on this open question.
\subsection{Existence} 
In this subsection we give the construction to prove the existence of each case found in Section \ref{Section-Extensions}. 

\begin{thm}
    Each case of Proposition \ref{Order 2 cases} and Proposition \ref{Order 3 cases} occurs as a pair of $(X,G)$ with $X$ a IHS manifold of K3$^{[2]}$-type and $G \subset \Aut(X)$.
\end{thm}
\begin{proof}
    Let $\Lambda=\Lambda_2$ be the lattice \eqref{Lattice K3^2} for $n = 2$, i.e.
    \[\Lambda=U^{\oplus3}\oplus E_8(-1)^{\oplus 2}\oplus\langle-2\rangle\]By \cite{HM} there is an embedding of $S^{G_0}(\Lambda)$ into $\Lambda$ such that $S^{G_0}(\Lambda)^{\perp_\Lambda}$ is isometric to $S_{G_0}(\Lambda)$. We can consider lattices $T(\Lambda)$ and $L(\Lambda)$ in $S^{G_0}(\Lambda)$ isometric to $T_X$ and $L$ respectively as in Proposition \ref{Order 2 cases} or \ref{Order 3 cases}. If $T(\Lambda) \oplus L(\Lambda)$ is not isometric to $S^{G_0}(\Lambda)$, then it is isometric to a sublattice of index 2, but in both cases $T^{\perp_\Lambda} \cong S_{G_0}(\Lambda)$, thus by Theorem \ref{MongardiExistence} 
    there is a marked IHS manifold $(X,\eta)$ with $G_0 \subset \Aut(X)$ acting simplectically on $X$ via $\eta$. We can consider the marking such that $T_X \cong T(\Lambda)$ and $S_{G_0}(\Lambda) \cong S_{G_0}(X)$. Furthermore we have that $S^{G_0}(X) \cong S^{G_0}(\Lambda)$ and we can consider $L := \eta^{-1}(L(\Lambda)) \subset S^{G_0}(X)$.

Now, we want to construct the respective non-symplectic automorphism for each $X$. The case of Proposition \ref{Order 3 cases} is analogous thus we will focus on Proposition \ref{Order 2 cases}. If $T := T_X \oplus L$ is not isometric to $S^{G_0}(X)$, there is an isotropic subgroup $H_{S^{G_0}(X)} \leq A_T$ of order 2 associated to $S^{G_0}(X)$. 
By Lemma \ref{ActionTX}, the action on $T_X$ has to be by roots of unity and preserving $\w_X$. Therefore we want to consider the following action on $T$:
\[ \varphi := -id_{2\times2} \oplus id_{1\times1} \curvearrowright T_X \oplus L  \]
i.e. the identity on $L$ and $-id$ on $T_X$. After computing $H_{S^{G_0}(X)}$ in $A_T$, we have to check that $\bar{\varphi}(H_{S^{G_0}(X)}) = H_{S^{G_0}(X)}$ so we can use Lemma \ref{ExtendOver} to extend this automorphism to $S^{G_0}(X)$. We denote by $\varphi$ again the extension of $\varphi$ to $S^{G_0}(X)$. 
\begin{rem}
    If $S^{G_0}(X) = T_X \oplus L$, then we just define $\varphi$ directly. This is the case for (1) in Proposition \ref{Order 2 cases}.
\end{rem}
We can then check that $L$ is fixed by the action of $\langle G_0, \varphi \rangle$ on $S^{G_0}(X)$.

Now, we want to extend our automorphism $\varphi$ from $S^{G_0}(X)$ to $H^2(X,\Z)$ using \cite[Proposition 1.5.1]{Nikulin}. To extend $\varphi$ from $S^{G_0}(X)$ to $H^2(X,\Z)$ we need to consider a gluing morphism $\gamma: H_L(X) \to A_{S^{G_0}(X)}$ where $H_L$ is a subgroup of order 81 of $A_{S_{G_0}}(X) = \Z_3 \times \Z_3 \times \Z_9$. In particular we have that $H_L = A_{S_{G_0}}(X)$. Since the morphism $\OO(S_{G_0}(X)) \to \OO(A_{S_{G_0}(X)})$ is surjective (see \cite[Section 7]{HM} and \cite[Table 9]{HM}), different choices of $\gamma$ produce isomorphic embeddings of $S^{G_0}(X)$ into $H^2(X,\Z)$ with orthogonal lattice $S_{G_0}(X)$.

Having the gluing morphism $\gamma$ we use Lemma \ref{ExtendComp}. For this we need to check if $\gamma$ satisfies: 
\[ \gamma \circ \bar{\psi} = \bar{\varphi} \circ \gamma \]
If $\varphi$ is extendable, then we can consider $G = \langle G_0,\varphi \rangle \subset \OO(H^2(X,\Z))$. Since the morphism
\[ \nu: \Aut(X) \to \OO(H^2(X,\Z)) \]
is injective, we can consider $G$ acting on $X$ with $G_0$ simplectically and $G/G_0 \simeq \langle \varphi \rangle$ by construction and Lemma \ref{ActionTX}.

We apply this method for each case of Proposition \ref{Order 2 cases} and Proposition \ref{Order 3 cases}, checking that the respective automorphism exists and thus proving that each case occurs.

Computations for the proof have been done with MAGMA and using part of codes contained in \cite{HM, W}.
 We now give the details of the computations for the case $L^2 = 6$ with $L = e-f$. The other cases are similar and we give in Table \ref{OtherCases} the necessary information to compute them.
 
When $L=e-f$, $T_X$ has generators $\{e+f, h\}$. Let $T = T_X \oplus L$. We have that 
\[ S^{G_0}(X)/T = \langle f + T \rangle = \langle e + T \rangle  \]
is an order 2 group such that $e+f \in T$. Moreover
\[ A_T = \Big\langle \frac{e+f}{18} + T,\, \frac{f}{3} + T,\, \frac{h}{6} + T \Big\rangle. \]
Then, by Lemma \ref{Overlattices}, there is a subgroup $H_{S^{G_0}(X)}$ in $A_T$ which can be easily seen as 
\[H_{S^{G_0}(X)} = \Big\langle \frac{e+f}{2} + T \Big\rangle.\] Let us consider the following action on $T$:
\[ \varphi := -id_{2\times2} \oplus id_{1\times1} \curvearrowright T_X \oplus L\]
and observe that \[\bar\varphi\left(\frac{e+f}{2}\right) =-\frac{e+f}{2}.\] 
Since $e+f \in T$, we have that $\bar\varphi(H_{S^{G_0}(X)}) = H_{S^{G_0}(X)}$, and thus by Lemma \ref{ExtendOver}, we can extend $\varphi$ to $S^{G_0}(X)$.
Then $\varphi$ is the isometry of $S^{G_0}(X)$ fixing $L = e-f$ represented by the matrix
\[ \varphi = \begin{pmatrix} 0 & -1 & 0 \\ -1 & 0 & 0 \\ 0 & 0 & -1 \end{pmatrix}.\]
Let $\{\bar{g_1},\bar{g_2},\bar{g_3}\}$ be a set of generators of $A_{S_{G_0}(X)}$. 
For the dual of $S^{G_0}(X)$, we take the basis
\[ S^{G_0}(X)^\vee = \Big\langle f_1,f_2,f_3 \Big\rangle \]
where
$f_1 = \frac{1}{3}(2e-f-h),\  f_2 = \frac{-2e}{3} + f,\ f_3=\frac{1}{9}(2e-f)-\frac{h}{6}$,
and such that their corresponding images by the quotient projection, denoted by $\{\bar{f_1}, \bar{f_2}, \bar{f_3}\}$, generate $A_{S^{G_0}(X)}$. 

To extend $\varphi$ from $S^{G_0}(X)$ to $H^2(X,\Z)$ we will consider the gluing morphism $\gamma : A_{S_{G_0}(X)} \to A_{S^{G_0}(X)}$ represented as a matrix in the generators of $A_{S_{G_0}(X)}$:
\[ \gamma = \begin{pmatrix} 0 & 1 & 0 \\ 1 & 0 & 0 \\ 0 & 0 & 2 \end{pmatrix}.\]
We need to find an isometry $\psi \in \OO(S_{G_0}(X))$ such that
\[  \gamma \circ \bar{\psi} =\bar{\varphi} \circ \gamma \]
where $\bar{\psi}$ is the corresponding morphism in $\OO(A_{S_{G_0}(X)})$. In other words, we can compute $\bar{\psi}$ as
\[ \bar{\psi} = \gamma^{-1} \circ \bar{\varphi} \circ \gamma. \]
For this, we can compute directly the action of $\varphi$ on the image of the generators by $\gamma$ using their corresponding pre-image on $S^{G_0}(X)^\vee$. We have 
\begin{align*}
    \varphi(f_1) &= \varphi\left(\frac{1}{3}(2e-f-h)\right) = \frac{1}{3}(e-2f+h)\\
    \varphi(f_2) &= \varphi\left(\frac{-2}{3}e+f\right) = -e+\frac{2}{3}f\\
    \varphi(2 f_3) &= \varphi\left(\frac{1}{9}(4e-2f)-\frac{4}{3}h\right) = \frac{2}{9}(e-2f)+\frac{4}{3}h.
\end{align*}
which implies 
\begin{align*}
    &\bar{\varphi}(\bar{f_1}) = 2\bar{f_1} \\
    &\bar{\varphi}(\bar{f_2}) = \bar{f_2}+12\bar{f_3} \\
    &\bar{\varphi}(2\bar{f_3}) = \bar{f_2}+4\bar{f_3}.
\end{align*}
Then, for each generator of $A_{S_{G_0}(X)}$ we have that
\begin{align*}
    \bar\psi(\bar{g_1}) &= (\gamma^{-1} \circ \bar\varphi \circ \gamma)(\bar{g_1}) = ( \gamma^{-1} \circ \bar\varphi)(\bar{f_2}) = \gamma^{-1}(\bar{f_2}+12\bar{f_3}) = \bar{g_1}+6\bar{g_3}. \\
    \bar\psi(\bar{g_2}) &= (\gamma^{-1} \circ \bar\varphi \circ \gamma)(\bar{g_2}) = ( \gamma^{-1} \circ \bar\varphi)(\bar{f_1}) = \gamma^{-1}(2\bar{f_1}) = 2\bar{g_2}. \\
    \bar\psi(\bar{g_3}) &= (\gamma^{-1} \circ \bar\varphi \circ \gamma)(\bar{g_3}) = ( \gamma^{-1} \circ \bar\varphi)(2\bar{f_3}) = \gamma^{-1}(\bar{f_2}+4\bar{f_3}) = \bar{g_1}+2\bar{g_3}.
\end{align*}
which can be resumed as the following matrix acting on $A_{S_{G_0}(X)}$:
\[ \bar{\psi} = \begin{pmatrix} 1 & 0 & 1 \\ 0 & 2 & 0 \\ 6 & 0 & 2 \end{pmatrix}. \]
Then we check that it belongs to $\OO(A_{S_{G_0}(X)})$. 
Since the morphism $\OO(S_{G_0}(X)) \to \OO(A_{S_{G_0}(X)})$ is surjective as mentioned before, there exists $\psi \in \OO(S_{G_0}(X))$. Thus, by Lemma \ref{ExtendComp} we can extend $\varphi$ to $H^2(X,\Z)$ such that it acts as $\psi$ on $S_{G_0}(X)$. The existence of the respective IHS manifold has been explained before. 

For the other cases we give in Table \ref{OtherCases} the respective isometry of $S^{G_0}(X)$ and gluing morphism.
\end{proof}

Having proved the existence, we are only missing computing the divisibility of the polarization to characterize each case. We observe that the case $L^2 = 6$ is the only case where ${\rm div}(L)$ can be different from 1 (see \cite{Debarre}). In fact the following holds:

\begin{lem}\label{LemmaDiv}
    Let $L^2=6$. If $L = h$, then ${\rm div}(L) = 2$, while if $L=e-f$ and $L=e$, then ${\rm div}(L) = 1$.
\end{lem}
\begin{proof}
    First, we remark that the divisibility of $L$ depends on the possible primitive embeddings of $L \oplus T_X$ into $H^2(X,\Z)$, and thus, of $S^{G_0}(X)$ into $H^2(X,\Z)$.

    The divisibility can be computed using the  characterization given in \cite[Lemma 2.1]{OG6} and adapted to our case:
        \[ {\rm div}(v,L) = \max\{d \in \N : \frac{v}{d} \in \gamma(A_{S_{G_0}(X)})^\perp\}. \]
    Thus, once we know the gluing morphism, we can compute directly the divisibility for each case. In Table \ref{OtherCases} we compiled the necessary information for the corresponding embedding for every case. 
    In particular, let us compute the divisibility when $L = h$. Since $\frac{h}{2} \in S^{G_0}(X)$ we just have to check that $\frac{h}{2} \in \gamma(A_{S_{G_0}(X)})^\perp$. We have that 
    \[ \gamma(A_{S_{G_0}(X)}) = \langle f_1, f_2, f_1 + 2f_3 \rangle \simeq \langle f_1, f_2, 2f_3 \rangle \] 
    and thus $\frac{h}{2} \in \gamma(A_{S_{G_0}(X)})^\perp$. This implies that ${\rm div}(L) = 2$. The other cases are similar.
\end{proof}

\begin{table}[H]
    \centering
    \begin{tabular}{c|c|c|c|c|c}
        Polarization & Generators of & Isometry & Order of  & Gluing Morphism & Isometry \\ 
        $L$&$T_X$&$\varphi\in\OO(S^{G_0}(X))$& Isometry & $\gamma$&$\bar\psi\in\OO(A_{S_{G_0}(X)})$\\
        \hline \hline
        $L=h$ & $\{e,f\}$&$\begin{pmatrix} -1& 0 & 0 \\ 0 & -1 & 0 \\ 0 & 0 & 1 \end{pmatrix}$ & 2 &
        $\begin{pmatrix} 0 & 1 & 1 \\ 1 & 0 & 0 \\ 0 & 0 & 2 \end{pmatrix}$ &
        $\begin{pmatrix} 2 & 0 & 0 \\ 0 & 1 & 1 \\ 0 & 6 & 2 \end{pmatrix}$
        \\
        \hline
        $L=e-f$ & $\{h,e+f\}$&
        $\begin{pmatrix} 0 & -1 & 0 \\ -1 & 0 & 0 \\ 0 & 0 & -1 \end{pmatrix}$ & 2 &
        $\begin{pmatrix} 0 & 1 & 0 \\ 1 & 0 & 0 \\ 0 & 0 & 2 \end{pmatrix}$ &
        $\begin{pmatrix} 1 & 0 & 1 \\ 0 & 2 & 0 \\ 6 & 0 & 2 \end{pmatrix}$\\

        \hline
        $L=e$ & $\{h,e-2f\}$&
        $\begin{pmatrix} 1 & 1 & 0 \\ 0 & -1 & 0 \\ 0 & 0 & -1 \end{pmatrix}$ & 2 &
        $\begin{pmatrix} 0 & 1 & 1 \\ 2 & 0 & 0 \\ 0 & 0 & 2 \end{pmatrix}$ & $\begin{pmatrix} 1 & 0 & 1 \\ 0 & 2 & 0 \\ 0 & 0 & 8 \end{pmatrix}$
        \\
        \hline
        $L=e+f$ &$\{e-f,h\}$&
        $\begin{pmatrix} 0 & 1 & 0 \\ 1 & 0 & 0 \\ 0 & 0 & -1 \end{pmatrix}$ & 2 &
        $\begin{pmatrix} 0 & 2 & 1 \\ 1 & 0 & 0 \\ 0 & 0 & 2 \end{pmatrix}$ & $\begin{pmatrix} 2 & 0 & 2 \\ 0 & 2 & 1 \\ 3 & 6 & 4 \end{pmatrix}$
        \\
        \hline
        $L=2e-f$ &$\{f,h\}$&
        $\begin{pmatrix} 1 & 0 & 0 \\ -1 & -1 & 0 \\ 0 & 0 & -1 \end{pmatrix}$ & 2 &
        $\begin{pmatrix} 1 & 1 & 1 \\ 0 & 1 & 0 \\ 6 & 0 & 2 \end{pmatrix}$ & $\begin{pmatrix} 1 & 0 & 2 \\ 0 & 2 & 0 \\ 0 & 0 & 1 \end{pmatrix}$ \\
        \hline
        $L=h$ & $\{e,f\}$&$\begin{pmatrix} -1 & -1 & 0 \\ 1 & 0 & 0 \\ 0 & 0 & 1 \end{pmatrix}$ & 3 &
        $\begin{pmatrix} 1 & 1 & 0 \\ 2 & 1 & 0 \\ 0 & 0 & 2 \end{pmatrix}$ &
        $\begin{pmatrix} 1 & 0 & 2 \\ 0 & 1 & 1 \\ 6 & 3 & 7 \end{pmatrix}$
        \\
        \hline
        $L=h$ & $\{e,f\}$&$\begin{pmatrix} 1 & 1 & 0 \\ -1 & 0 & 0 \\ 0 & 0 & 1 \end{pmatrix}$ & 6 &
        $\begin{pmatrix} 0 & 1 & 1 \\ 1 & 0 & 0 \\ 0 & 0 & 2 \end{pmatrix}$ &
        $\begin{pmatrix} 2 & 0 & 1 \\ 0 & 1 & 1 \\ 6 & 6 & 5 \end{pmatrix}$
        \\
    \end{tabular}
    \caption{Data for all cases}
    \label{OtherCases}
\end{table}
\bibliographystyle{aomalpha}
\bibliography{references}
\end{document}